\documentclass[11pt, a4paper, reqno]{amsart}

\usepackage[margin=2.5cm]{geometry}

\usepackage[utf8]{inputenc}
\usepackage[T1]{fontenc}
\usepackage{color}
\usepackage[foot]{amsaddr}

\usepackage[colorlinks, linkcolor=blue]{hyperref}
\usepackage{xurl} 
\hypersetup{breaklinks=true} 

\usepackage[url=false, style=numeric, backend=biber,firstinits=true, maxbibnames=99, maxalphanames=4]{biblatex}
\DeclareFieldFormat{issn}{}
\usepackage[capitalize]{cleveref}
\crefname{enumi}{}{}
\crefname{equation}{}{}

\newtheorem{theorem}{Theorem}

\newenvironment{delayedproof}[1]
 {\begin{proof}[\raisedtarget{#1}Proof of \Cref{#1}]}
 {\end{proof}}
\newcommand{\raisedtarget}[1]{%
  \raisebox{\fontcharht\font`P}[0pt][0pt]{\hypertarget{#1}{}}%
}

\newtheorem{proposition}{Proposition}[section]
\newtheorem{lemma}[proposition]{Lemma}

\theoremstyle{definition}
\newtheorem{definition}[proposition]{Definition}
\newtheorem{remark}[proposition]{Remark}

\numberwithin{equation}{section}

\newcommand{\nocontentsline}[3]{}
\newcommand{\tocless}[2]{\bgroup\let\addcontentsline=\nocontentsline#1{#2}\egroup}

\usepackage{amssymb, mathrsfs}
\def \R {\mathbb{R}}
\def \rmH {\mathrm{H}}
\def \rmd {\mathrm{d}}
\def \L {\mathscr{L}}
\def \K {\mathscr{K}}
\def \rmL {\mathrm{L}}

\def \calU {\mathcal{U}}
\def \calV {\mathcal{V}}
\def \calA {\mathcal{A}}
\def \bfB {\mathbf{B}}
\def \bfC {\mathbf{C}}
\def \rmV {\mathrm{V}}
\def \rmW {\mathrm{W}}

\usepackage{mathtools}
\DeclarePairedDelimiter\abs{\lvert}{\rvert}
\DeclarePairedDelimiter\norm{\lVert}{\rVert}

\DeclarePairedDelimiter{\pair}{\langle}{\rangle}

\makeatletter

\newcommand{\lefteqno}{\let\veqno\@@leqno}
\renewcommand{\eqref}[1]{\textup{(\ignorespaces\ref{#1}\unskip\@@italiccorr)}}

\makeatother

\begin{document}


\title[Boundedness of weak solutions to degenerate Kolmogorov equations]{Boundedness of weak solutions to degenerate Kolmogorov equations of hypoelliptic type in bounded domains}

\author{Mingyi Hou}
\address{Mingyi Hou\\Department of Mathematics, Uppsala University\\
751 05 Uppsala, Sweden}
\email{mingyi.hou@math.uu.se}

\date{\today}

\begin{abstract}
    We establish the boundedness of weak subsolutions for a class of degenerate Kolmogorov equations of the hypoelliptic type, compatible with a homogeneous Lie group structure, within bounded product domains using the De Giorgi iteration. We employ the renormalization formula to handle boundary values and provide energy estimates. An $\rmL^1$--$\rmL^p$ type embedding estimate derived from the fundamental solution is utilized to incorporate lower-order divergence terms. This work naturally extends the boundedness theory for uniformly parabolic equations, with matching exponents for the coefficients.
\end{abstract}

\maketitle

\tableofcontents

\section{Introduction and the main result}
Let $N \geq 2$ and $0 < m_0 \leq N$ be positive integers, and let $z=(x,t)\in\R^N\times\R$ denote a point. The (backward) Fokker-Planck-Kolmogorov equation of divergence form, which is degenerate if $m_0 < N$, for a function $u(x,t)$, is given by: 
\begin{equation}\label{eq:kol}
    D_t u - \langle \mathbf{B}x, Du \rangle = \L u + g + D_i f^i,
\end{equation}
where, unless otherwise stated, the summation for indices $i,j$ is from $1$ to $m_0$. 
Here, $\bfB$ is a constant real matrix, $D = (D_1, \dots, D_N)$ is the gradient, and $\langle\cdot,\cdot\rangle$ is the Euclidean inner product. 
The functions $g(x,t)$ and $f^i(x,t)$ are measurable. 
The operator $\L$ is defined as:  
\begin{equation*}
    \L u =  D_i \left( a^{ij}(x,t) D_j u + b^i(x,t) u \right) +  c^i(x,t) D_i u + d(x,t) u
\end{equation*}
where $a^{ij}, b^i, c, d$ (for $i,j=1,\dots,m_0$) are measurable functions, and $a^{ij} = a^{ji}$. We may also denote the Kolmogorov operator by $ \K := \mathscr{T} - \L $, where $\mathscr{T} = D_t - \langle \bfB x, D \rangle$ is the transport part.

Throughout this paper, we assume the following for some positive constants $\lambda$ and $\Lambda$:
\begin{equation}\label{assump:elliptic}
    \lefteqno
    \tag{\textrm{H1}}
    \abs{a^{ij}(x,t)} \leq \Lambda  \textrm{ and } a^{ij}(x,t) \xi_i \xi_j \geq \lambda \abs{\xi}^2, \quad\forall(x,t)\in\R^{N+1},\, \xi\in\R^{N}.
\end{equation}
Additionally, we assume that the matrix $\bfB$ has the following form: 
\begin{equation}\label{assump:B}
    \lefteqno
    \tag{\textrm{H2}}
    \bfB = 
    \begin{pmatrix}
        \mathbf{O} &\mathbf{O} & \cdots & \mathbf{O} & \mathbf{O} \\
        \bfB_1 & \mathbf{O} & \cdots & \mathbf{O} & \mathbf{O} \\
        \mathbf{O} & \bfB_2 & \cdots & \mathbf{O} & \mathbf{O} \\
        \vdots & \vdots & \ddots & \vdots & \vdots \\
        \mathbf{O} & \mathbf{O} & \cdots & \bfB_{\kappa} & \mathbf{O}
    \end{pmatrix}
\end{equation}
where $\mathbf{O}$ is the zero matrix and each $\bfB_j$ is a $m_j \times m_{j-1}$ matrix of rank $m_j$, with $j=1,\dots,\kappa$ and $m_j$ being positive integers such that:
\begin{equation*}
    m_0 \geq m_1 \geq \dots \geq m_\kappa \geq 1, \textrm{ and } m_0 + m_1 + \cdots + m_\kappa = N. 
\end{equation*}
The homogeneous dimension is defined as: 
\begin{equation*}
    Q := m_0 + 3m_1 + \cdots + (2\kappa + 1)m_\kappa.
\end{equation*}

To present our main result, we introduce the necessary function space preliminaries. 
Throughout this paper, we consider $\Omega = \calV \times \calU$, where $\calV \subset \R^{m_0}$ and $\calU \subset \R^{N-m_0}$ are bounded domains, with $\partial\calV$ being $\mathrm{C}^{0,1}$ and $\partial\calU$ being $\mathrm{C}^{1,1}$.
We assume $N \geq 2$, or equivalently $Q \geq 2$. 
The time cylinder is defined as $\Omega_T := \Omega \times (0,T)$. 

We define the function space $\rmH^1_{\mathrm{kin}}(\Omega_T)$ as follows: 
\begin{equation*}
    \rmH^1_{\mathrm{kin}}(\Omega_T) := \{ u(x,t) \in \rmL^2(\calU_T; \rmH^1(\calV)) \textrm{ such that } \langle \bfB x, Du \rangle \in \rmL^2(\calU_T; \rmH^{-1}(\calV)) \},
\end{equation*}
where $\calU_T := \calU \times (0,T)$. 
This space is equipped with the norm 
\begin{equation*}
    \norm{u}_{\rmH^1_{\mathrm{kin}}}^2 = \norm{u}^2_{\rmL^2(\calU_T; \rmH^1(\calV))} + \norm{\langle \bfB x, Du \rangle}_{\rmL^2(\calU_T; \rmH^{-1}(\calV))}^2,
\end{equation*}
where $\rmH^{-1}(\calV)$ is the dual space of $\rmH^1_0(\calV)$. 

Next, we define the space $\rmV^0_{\mathrm{kin}}(\Omega_T)$ as  
\begin{equation*}
    \rmV^0_{\mathrm{kin}}(\Omega_T) := \rmH^1_{\mathrm{kin}}(\Omega_T) \cap \mathrm{C}([0,T]; \rmL^2(\Omega))
\end{equation*}
equipped with the norm 
\begin{equation*}
    \norm{u}_{\rmV^0_{\mathrm{kin}}(\Omega_T)} := \norm{u}_{\rmH^1_{\mathrm{kin}}(\Omega_T)} + \sup_{0<t<T} \norm{u(\cdot, t)}_{\rmL^2(\Omega)}.
\end{equation*}

{ The weak trace was initially observed in \cite{Sil22} and later formally introduced in \cite{AH24}.}
\begin{equation*}
    \mathrm{tr}_{\Gamma_{\mathrm{K}}} : \rmV^0_{\mathrm{kin}}(\Omega_T) \to \rmL^2_{\mathrm{loc}}(\Gamma_{\mathrm{K}}, \abs{\langle \bfB x, \mathbf{n}_x \rangle}^2),
\end{equation*}
{  $\Gamma_{\mathrm{K}} := \calV \times (\partial\calU) \times (0,T)$, 
$\mathbf{n}_x$ denotes the outward unit normal for $\Omega$, and $\abs{\pair{\bfB x,\mathbf{n}_x}}^2$ is the weight. It is important to note that for a classical trace operator we expect the weight to be $\abs{\pair{\bfB x, \mathbf{n}_x}}$, i.e.~without the square. However, it is still an open problem if such a trace exists. The difference between two traces is that, for the weak trace, there is not an integration by parts formula for two functions belonging to $\rmH_{\mathrm{kin}}^1(\Omega_T)$. Thus, directly testing the equation against its own solution is not feasible. Nevertheless, using the weak trace we can test the equation against nice test functions to get a renormalization formula, \cref{thm:renormal}, which serves as a substitute for classical energy estimates.}
Throughout this paper, we denote $u|_{\Gamma_{\mathrm{K}}} = \mathrm{tr}_{\Gamma_{\mathrm{K}}}(u)$, implying that boundary values are understood in the weak trace sense.

\begin{definition}
    A function $u \in \rmV^0_{\mathrm{kin}}(\Omega_T)$ is called a subsolution to \eqref{eq:kol} if for all $v \in \mathrm{C}_c^1(\Omega_T)$ with $v \geq 0$, the inequality 
    \begin{equation*}
        \iint_{\Omega_T} \left(-u v_t + u \langle \bfB x, Dv \rangle + \L[u,v] - gv + f^i D_i v \right) \,\mathrm{d}x \mathrm{d}t \leq 0
    \end{equation*}
    holds, where 
    \begin{equation*}
        \L[u,v] := a^{ij} D_j u D_i v + b^i u D_i v - c^i (D_i u) v - duv.
    \end{equation*}
\end{definition}

Now, we state the main result. 
We further separate the boundary of $\Omega_T$ into parts: 
\begin{align*}
    &\Gamma_{\mathrm{P}} := (\partial\calV \times \calU \times [0,T]) \cup (\Omega \times \{0\}),\\
    &\Gamma^+_{\mathrm{K}} := \{(x, t) \in \Gamma_{\mathrm{K}}: \langle \bfB x, \mathbf{n}_x \rangle \geq 0\},\\
    &\Gamma_{\mathrm{K}}^-:= \{(x,t)\in\Gamma_\mathrm{K}: \pair*{\bfB x, \mathbf{n}_x}<0\}.
\end{align*}
Then we define 
\begin{equation*}
    M := \sup_{\Gamma_\mathrm{K}^+\cup \Gamma_\mathrm{P}} u_+,
\end{equation*}
where $u_+ := \max\{u, 0\}$ denotes the positive part of a function. 
Define
\begin{equation*}
    p_0 := \frac{Q+2}{Q} \textrm{ and } q_0 := \frac{Q+2}{2}
\end{equation*}
the optimal embedding constants.
Assume $\tilde{q} > q_0$ is fixed, and consider either:
\begin{equation}
    \lefteqno
    \tag{\textrm{Data 1}}
    \label{eq:data1}
    c^i \in \rmL^{2\tilde{q}}(\Omega_T),\, d, g \in \rmL^{\tilde{q}}(\Omega_T),\, b^i, f^i \in \rmL^{2\tilde{q}}(\Omega_T);
\end{equation}
or
\begin{equation}
    \lefteqno
    \tag{\textrm{Data 2}}
    \label{eq:data2}
    c^i \in \rmL^{Q+2}(\Omega_T),\, d, g \in \rmL^{\tilde{q}}(\Omega_T),\, b^i, f^i = 0. 
\end{equation}
Moreover, $\norm{\cdot}_{p,\Omega_T}$ denotes the $\rmL^p(\Omega_T)$ norm.

\begin{theorem}\label{thm:1}
    Let $\Omega = \calV \times \calU\subset\R^{m_0}\times\R^{N-m_0}$ be a bounded product domain with $\partial\calV$ being $\mathrm{C}^{0,1}$ and $\partial\calU$ being $\mathrm{C}^{1,1}$, and let the assumptions \cref{assump:elliptic,assump:B} hold. 
    If $u\in\rmV^0_{\mathrm{kin}}(\Omega_T)$ is a subsolution to \cref{eq:kol}, and either \cref{eq:data1} or \cref{eq:data2} holds, 
    then there exists a constant $C = C (\lambda, \Lambda, \mathrm{Data})$ such that 
    \begin{equation}\label{eq:linfl2}
        \sup_{\Gamma_{\mathrm{K}}^-} u,\, \sup_{\Omega_T} u \leq M + C \max\{1, M,\norm{(u-M)_+}_{2,\Omega_T}\},
    \end{equation}
    and 
    \begin{equation}\label{eq:linfdata}
        \sup_{\Gamma_{\mathrm{K}}^-} u,\, \sup_{\Omega_T} u \leq (1+ C) \max\{1,M\}. 
    \end{equation}
\end{theorem}

\begin{remark}
    The first inequality provides an $\rmL^2$--$\rmL^\infty$ estimate which can be made local and is crucial for the Harnack inequality. 
    The renormalization formula, { \cref{thm:renormal}}, and an $\rmL^1$--$\rmL^{p_0}$ type embedding estimate, { \cref{thm:embedding:l1}}, make it possible to include divergence terms $b^i$ and $f^i$.
    The second inequality is a classical boundedness estimate up to the boundary. 
    Therefore the current theory is a canonical extension of the classical boundedness theory for uniformly parabolic equations.  
\end{remark}

\begin{remark}
    Consider the case $m_0 = N \geq 2$, the Kolmogorov equation \cref{eq:kol} becomes the uniformly parabolic equation. 
    In this case, $Q = N$, and under \cref{eq:data2}, the optimal exponent $\frac{N+2}{2}$ from the boundedness theory for uniformly parabolic equations is recovered, see \cite{LSU68, DiB93}. 
    Under \cref{eq:data1}, the exponent is still optimal for the terms $b^i, d, f^i, g$, but we cannot reach $N+2$ for the coefficient $c^i$ due to the $\rmL^1$--$\rmL^{p_0}$ embedding, see \cref{thm:embedding:l1}. 
    Therefore, we believe our result is optimal for the degenerate Kolmogorov equations of hypoelliptic type. 
\end{remark}

An immediate corollary is the weak maximum principle.
\begin{theorem}\label{thm:2}
    Let $\Omega$ be defined as in \cref{thm:1}, and let the assumptions \cref{assump:elliptic,assump:B} hold. 
    Assume that $b^i, c^i\in \rmL^{Q+2}(\Omega_T)$, $d\in\rmL^{2}(\Omega_T)$, $g,f^i=0$, and for all $v\in \mathrm{C}_c^1(\overline{\Omega_T})$ with $v\geq0$ and $v=0$ on $(\partial\calV)\times\calU\times(0,T)$ that
    \begin{equation*}
        \int_{\Omega_T} dv - b^i D_i v\,\rmd z \leq 0.
    \end{equation*}
    If $u\in\rmV^0_{\mathrm{kin}}(\Omega_T)$ is a subsolution to \cref{eq:kol}, then 
    \begin{equation*}
        \sup_{\Gamma_{\mathrm{K}}^-} u,\, \sup_{\Omega_T} u \leq M. 
    \end{equation*}
\end{theorem}

\subsection{A brief survey of the question}
Kolmogorov equations of hypoelliptic type, also known as ultraparabolic equations, have broad applications in various fields, for instance, statistical physics and mathematical finance.
It has been extensively studied and a comprehensive overview can be found in \cite{AP20}.

Recent attention has been directed towards the weak solution theory for Kolmogorov equations with rough coefficients.
Significant progress has been made:
The Harnack inequality for weak solutions to the kinetic Fokker-Planck was initially established in \cite{GIMV19}, followed by an alternative proof in \cite{GI23} and a quantitative Harnack inequality \cite{GM22}.
Additionally, \cite{AR22} generalizes the Harnack inequality to degenerate Kolmogorov equations. 
Meanwhile, the function space $\rmH^1_{\mathrm{kin}}$, which is suitable for a weak solution theory for the kinetic Fokker-Planck equation, was proposed in \cite{AAMN21} and stimulated further interest in the field. 
The existence of weak solutions has been well-established in various contexts, e.g.~in domains without boundaries \cite{AAMN21,AHN24,AIN24weak}, and in bounded domains \cite{LN21,Zhu22,GN23,AH24}.

However, proceeding further with the weak solution theory in bounded domains encounters challenges, notably the trace problem as highlighted in \cite{AAMN21}.
Specifically, the classical trace for the function space $\rmH^1_{\mathrm{kin}}(\Omega_T)$ remains open.
This issue has prompted consideration of a weaker notion of trace, first proposed for local boundary regularity in \cite{Sil22} and formally applied to the function space $\rmV^0_{\mathrm{kin}}(\Omega_T)$ in \cite{AH24}.

A crucial complement to the weak trace concept is the renormalization formula, e.g.~\cref{thm:renormal}, which plays a key role in handling energy estimates tested against the weak solution itself.
The concept of renormalization was introduced by DiPerna and Lions in \cite{DL89}, extended by Mischler to Vlasov and other kinetic equations in bounded domains in \cite{Mis00,Mis10}, and applied to weak solutions to kinetic Fokker-Planck equations in \cite{Zhu22}.
Recently, the renormalization formula for functions in the space $\rmV^0_{\mathrm{kin}}$ is introduced in \cite{AH24}. 

Regarding the boundedness of weak solutions, it is known in \cite{PP04} that weak solutions are locally bounded based on Sobolev embeddings derived from the fundamental solution from \cite{Fol75}. 
This concept has been further developed in subsequent works, such as \cite{WZ09,WZ11,GIMV19,AR22,WZ24}, to obtain local boundedness which is the key component for both Harnack inequality and H\"older regularity. 
On the other hand, in \cite{GM22}, a slightly different embedding result (still based on the fundamental solution), which is also used in \cite{IS20}, was developed. 
The current work is inspired by \cite{GM22}. 
Our approach integrates the recent advancements in weak trace theory, renormalization techniques, and an $\rmL^1$--$\rmL^{p_0}$ embedding estimate. 
Notably, the global boundedness up to the boundary for weak solutions to degenerate Kolmogorov equations presented in this work is novel to the best of our knowledge. 
Furthermore, it is noteworthy that the renormalization technique utilized in this study has previously proven effective in enhancing regularity theory, as demonstrated in the context of the 
nonlinear elliptic equations with general measure data, see e.g.~\cite{DMM99}.

\subsection{Sketch of the proof and outline}
The underlying idea for the proof of \cref{thm:1} is analogous to its counterpart for uniformly parabolic equations, see \cite{LSU68,DiB93}, and also similar to the proof of local $\rmL^2$--$\rmL^\infty$ estimates appeared in the literature mentioned above. 
Essentially, Sobolev embeddings enable higher integrability of weak solutions. 
By using the information of the equation, in particular, the Caccioppoli estimate, this gain of integrability can be further improved by either Moser or De Giorgi iteration. 
In this work, we will present the De Giorgi iteration, see e.g.~\cref{thm:it1}. 

Here's an outline of the proof structure:
In \cref{sec:intkernel}, we establish the optimal integrability of the fundamental solution defined in \cref{eq:kernel} to the principal equation \cref{eq:principal}; 
In \cref{sec:embedding}, we utilize the integrability of the kernel to prove an $\rmL^1$--$\rmL^{p_0}$ embedding result, see \cref{thm:embedding:l1}, alongside a known $\rmL^2$--$\rmL^{2p_0}$ embedding, see \cref{thm:embedding:l2};   
In \cref{sec:renormal}, we introduce the renormalization formula \cref{eq:renormal}. 
Based on the Sobolev type estimates and the renormalization formula we can proceed with the iteration,
and we first prove it under \cref{eq:data1}. 
In \cref{sec:truncate}, we construct truncations $\Psi_{k,l}(u)$ of the undercut of the subsolution, see \cref{eq:psi}, and show $\rmL^{2p_0}$ integrability; 
In \cref{sec:energy}, we derive uniform energy estimates for $\Psi_{k,l}(u)$ and show $(u-k)_+$ is $\rmL^{2p_0}$ integrable in small time intervals;
In \cref{sec:caccioppoli}, we derive a Cappioppoli estimate \cref{eq:it1:caccioppoli};
In \cref{sec:iteration}, by combining the Cappioppoli estimate and the Sobolev embedding we obtain iteration inequalities
of the form of De Giorgi, see \cref{eq:it1:iteration}; 
In \cref{sec:proof1,sec:proof2} we run the iteration and prove \cref{eq:linfl2,eq:linfdata}.
Concerning \cref{eq:data2}, we sketch the proof in \cref{sec:proof:data2} which is essentially the same as the proof for \cref{eq:data1} except the Sobolev embedding, see \cref{thm:embedding:data2}. 
Finally, the proof of \cref{thm:2} is sketched in \cref{sec:remarks} together with some closing remarks.

\section{Intergrability, embedding and renormalization}\label{sec:pre}
\subsection{Integrability of the Kolmogorov kernel}\label{sec:intkernel}
It is well-known (see \cite{AP20} for an \allowbreak overview) 
that the following constant coefficient equation has a fundamental solution:
\begin{equation}\label{eq:principal}
    u_t - \langle \bfB x, Du \rangle = \L_0 u, 
\end{equation}
where $\L_0 := D_{i}^2 u $
is the Laplace operator. 
We denote $\K_0 := \mathscr{T}-\L_0$.

Indeed, if we define the Lie group structure: $(x,t)\circ (y,s) := (y+ \mathbf{E}(s)x, t+s)$ for $(x,t),(y,s)\in\R^{N+1}$, where $\mathbf{E}(t) = e^{-t\bfB}$,
then the fundamental solution to \cref{eq:principal} is given by 
$K(x,t; y,s) = K((y,s)^{-1}\circ (x,t))$,
where 
\begin{equation}\label{eq:kernel}
    K(x,t) = \frac{C_N}{t^{Q / 2}} \exp\left(-\frac{1}{4} \pair*{\bfC^{-1}(1)\delta_N \left(\frac{1}{\sqrt{t}}\right) x, \delta_N\left(\frac{1}{\sqrt{t}}\right)x} \right),
\end{equation}
in which $C_N = (4\pi)^{-\frac{N}{2}} \left|\det \bfC(1)\right|^{-\frac{1}{2}}$,
\begin{equation*}
    \bfC(t):= \int_0^t \mathbf{E}(s) \mathbf{A}_0 \mathbf{E}^T (s)\,\mathrm{d}s,\quad \mathbf{A}_0 := 
    \begin{pmatrix}
        \mathbf{I}_{m_0} & \mathbf{O} \\
        \mathbf{O} & \mathbf{O}
    \end{pmatrix},
\end{equation*}
and 
\begin{equation*}
    \delta_N(r) = \mathrm{diag} (r\mathbf{I}_{m_0}, r^3 \mathbf{I}_{m1}, \cdots, r^{2\kappa +1 }\mathbf{I}_{m_\kappa} ) \textrm{ for all } r>0.
\end{equation*}

To derive the embedding results for weak solutions, we need to estimate the integrability of the kernel $K$ {  in $\R^N\times(0,T)$ for $T>0$ a fixed time}. 
The following estimates, inspired by \cite{GM22}, suit our needs best.  
\begin{lemma}[Integrablity of the fundamental solution]\label{thm:intkernel}
    Let $K$ be the fundamental solution of \cref{eq:principal} defined as in \cref{eq:kernel}. Let $T>0$ be any fixed time. 
    Then $K\in\rmL^{p}(\R^{N}\times(0,T))$ for all $1\leq p< p_0$ where 
    \begin{equation*}
        p_0 = \frac{Q+2}{Q},
    \end{equation*} 
    and 
    $D_i K \in \rmL^{p}(\R^{N}\times(0,T))$, $i=1,\dots, m_0$, for all $1\leq p < p_1$ where 
    \begin{equation*}
        p_1 = \frac{Q+2}{Q+1}.
    \end{equation*}
\end{lemma}
\begin{proof}
    A direct calculation gives
    \begin{equation*}
        \int_0^T \int_{\R^{N}} \abs{K}^p \,\mathrm{d}x\rmd t = \int_0^T \int_{\R^{N}} \frac{(C_N)^p}{\abs{t}^{\frac{Q}{2}p}} \exp\left(-\frac{p}{4} \pair*{\mathbf{C}^{-1}(1) \delta_N \left(\frac{1}{\sqrt{t}}\right) x, \delta_N \left(\frac{1}{\sqrt{t}}\right)x}\right)\,\mathrm{d}x \rmd t.
    \end{equation*}
    By setting $y = \delta_N\left(\frac{1}{\sqrt{t}}\right)x$ we see 
    \begin{align*}
        \int_0^T \int_{\R^{N}} \abs{K}^p \,\mathrm{d}x\rmd t =& \int_0^T \int_{\R^{N}} \frac{(C_N)^p}{\abs{t}^{\frac{Q}{2}p}} \exp\left(-\frac{p}{4}\pair*{\mathbf{C}^{-1}(1) y, y} \right) \abs*{\det \delta_N\left(\frac{1}{\sqrt{t}}\right)}^{-1}\,\mathrm{d}y\mathrm{d}t\\
        =& (C_N)^p \int_{\R^N} \exp\left(-\frac{p}{4}\pair*{\mathbf{C}^{-1}(1) y, y} \right) \,\mathrm{d}y \int_0^T \abs{t}^{-\frac{Q}{2}p + \frac{Q}{2}}\,\mathrm{d}t .
    \end{align*}
    Hence, as long as 
    \begin{equation*}
        -\frac{Q}{2}p + \frac{Q}{2} >-1, \textrm{ which implies } p < \frac{Q+2}{Q}, 
    \end{equation*}
    $\abs{K}^p$ is integrable. 

    For each $D_i K$, $i=1,\dots, m_0$, it is clear from \cref{eq:kernel} that 
    \begin{equation*}
        \abs{D_i K} \leq \frac{C}{\sqrt{t}} K 
    \end{equation*}
    for some constant $C>0$. {  For the kernel of degenerate Kolmogorov operators in non-divergence form, similar observations can be found in \cite{DFA05}.}
    It follows from the above calculation that we need to choose $p$ such that 
    \begin{equation*}
        -\frac{Q}{2}p -\frac{1}{2}p + \frac{Q}{2}>-1, \textrm{ which implies } p < \frac{Q+2}{Q+1},
    \end{equation*}
    to ensure $\abs{D_i K}^p$ is integrable.
\end{proof}

\subsection{Some embedding results}\label{sec:embedding}
{  For $f,g\in \rmL^1 (\R^N\times(0,T))$, we define the convolution on the homogeneous Lie group $(\R^{N+1}, \circ)$ as follows: extend $f,g$ to be 0 on $\R^{N+1}\setminus (\R^N\times(0,T))$, 
and observe $(y,s)^{-1}\circ (x,t) := (x - \mathbf{E}(t-s)y, t-s)$, then for any $t>0$,
\begin{equation*}
    f*g(x,t):= \int_0^t \int_{\R^{N}} f((y,s)^{-1}\circ(x,t)) g(y,s) \,\mathrm{d}y\mathrm{d}s.
\end{equation*}}
It is known from \cite[Prop.~1.3.21]{BLU07} that the Lebesgue measure on $\R^{N+1}$ is invariant with respect to the left and the right translations on the Lie group $(\R^{N+1},\circ)$.
Thus, Young's convolution inequality applies, see e.g.~\cite[Lemma 1.4]{BCD11}.
\begin{lemma}[Young's convolution inequality]
    Suppose $1\leq p,q,r\leq \infty$ such that
    \begin{equation*}
        \frac{1}{p} +\frac{1}{q} = \frac{1}{r} + 1.
    \end{equation*}
    If $f\in\rmL^p(\R^{N}\times(0,T))$ and $g\in \rmL^q(\R^{N}\times(0,T))$, then $f*g$ exists almost everywhere and is in $\rmL^r(\R^{N+1})$, with
    \begin{equation*}
        \norm{f*g}_{r,\, \R^N\times(0,T)} \leq \norm{f}_{p,\, \R^N\times(0,T)} \norm{g}_{q,\, \R^N\times(0,T)}. 
    \end{equation*}
\end{lemma}
\begin{proof}
    First, observe that we always have $r\geq p,q$. 
    By the H\"older inequality, we have 
    \begin{align*}
        \abs{f * g(\xi)}  
        \leq & \left(\int_{\R^N\times(0,T)}\abs{f(\zeta^{-1}\circ\xi)}^p \abs{g(\zeta)}^q \rmd \zeta\right)^{\frac{1}{r}} \times \norm{f}_{p,\, \R^N\times(0,T)}^{(r-p)/r} \times\norm{g}_{q, \,\R^N\times(0,T)}^{(r-q)/r},
    \end{align*}
    where $\xi = (x,t)$ and $\zeta = (y,s)$.
    It follows from Fubini's theorem that
    \begin{align*}
        \norm*{f * g}_{r,\, \R^N\times(0,T)}^r =& \int_{\R^N\times(0,T)} \abs{f*g(\xi)}^r \,\rmd\xi\\
        \leq & \norm{f}_{p,\R^N\times(0,T)}^{r-p} \norm{g}_{q, \R^N\times(0,T)}^{r-q}\\
        &\times \int_{\R^N\times(0,T)} 
        \int_{\R^N\times(0,T)}\abs{f(\zeta^{-1}\circ\xi)}^p \abs{g(\zeta)}^q \,\rmd \zeta \rmd \xi\\
        \leq & \norm{f}_{p,\R^N\times(0,T)}^{r-p} \norm{g}_{q, \R^N\times(0,T)}^{r-q} \norm{f}_{p,\R^N\times(0,T)}^{p} \norm{g}_{q, \R^N\times(0,T)}^{q},
    \end{align*}
    which completes the proof. 
\end{proof}
{  We note that a local version of Young's convolution inequality, applicable when the underlying domain is only an open subset of $\R^{N+1}$, is presented in \cite{HLW24}.}

The next proposition is a direct consequence of \cref{thm:intkernel} and Young's convolution inequality.
\begin{proposition}[$\rmL^1$--$\rmL^p$ embedding]\label{thm:embedding:l1}
    Let $K$ be defined as in \cref{eq:kernel}, and let $T>0$ be any fixed time.
    If $u\in \rmL^q(\R^{N}\times(0,T))$ for some $1\leq q <\infty$, 
    then 
    \begin{equation*}
        \norm{K*u}_{p,\, \R^{N}\times(0,T)} \leq \sigma_0 \norm{u}_{q,\, \R^{N}\times(0,T)},
    \end{equation*}
    where $\sigma_0 = \norm{K}_{p_0-\varepsilon_0,\, \R^{N}\times(0,T)}$ for all $p\geq 1$ and $0<\varepsilon_0 \leq p_0 - 1$ satisfying
    \begin{equation*}
        \frac{1}{p} = \frac{1}{q} + \frac{1}{p_0 - \varepsilon_0} - 1.
    \end{equation*}

    Moreover, for any $1\leq q<\infty$ and $i=1,2,\dots,m_0$, we have 
    \begin{equation*}
        \norm{D_i (K* u)}_{p,\, \R^{N}\times(0,T)} \leq \sigma_1 \norm{u}_{q,\, \R^{N}\times(0,T)},
    \end{equation*}
    where $\sigma_1= \norm{D_i K}_{p_1-\varepsilon_1,\, \R^{N}\times(0,T)}$ for all $p\geq 1$ and $0<\varepsilon_1 \leq p_1 - 1$ satisfying
    \begin{equation*}
        \frac{1}{p} =  \frac{1}{q} + \frac{1}{p_1 - \varepsilon_1} - 1.
    \end{equation*}
\end{proposition}

{  Another commonly used embedding result, originally from \cite{Fol75} and cited in \cite{PP04}, is stated below.}
This result is sharp in the sense that it achieves $p_0$ and $p_1$. 
{  It is also a global estimate where the underlying domain in the time direction can be unbounded, since it uses the weak $\rmL^p$ norm for the kernel $K$. But it cannot be estimated in terms of the $\rmL^1$ norm. }
\begin{proposition}[$\rmL^2$--$\rmL^p$ embedding]\label{thm:embedding:l2}
    If $u\in \rmL^q(\R^{N+1})$ for any $1<q <\infty$, then there exists a constant $\varsigma_0$ such that 
    \begin{equation*}
        \norm{K* u}_{p,\, \R^{N+1}} \leq \varsigma_0 \norm{u}_{q,\,\R^{N+1}}
    \end{equation*}
    where 
    \begin{equation*}
        \frac{1}{p} = \frac{1}{q} - \frac{2}{Q+2}.
    \end{equation*}
    Similarly, for any $1<q<\infty$, there exists a constant $\varsigma_1$ such that
    \begin{equation*}
        \norm{D_i (K* u)}_{p,\, \R^{N+1}} \leq \varsigma_1 \norm{u}_{q,\,\R^{N+1}}
    \end{equation*}
    where 
    \begin{equation*}
        \frac{1}{p} =\frac{1}{q} -\frac{1}{Q+2}. 
    \end{equation*}
\end{proposition}
\begin{remark}
    These estimates are optimal:  \cref{thm:intkernel} coincides with the special case when $K$ is the heat kernel; 
    by taking $q=2$ in \cref{thm:embedding:l2}, we can recover the Sobolev embedding for the uniformly parabolic case, see \cite{LSU68,DiB93}. 
    Moreover, \cref{thm:intkernel} also coincides with the kinetic Fokker-Planck case as obtained in \cite[Lemma 10]{GM22} (in this case $N = 2m_0$ and $Q = m_0 + 3m_0 = 4m_0$). 
\end{remark}

\subsection{Renormalization of the weak solution}\label{sec:renormal}
In this section, we present the renormalization formula for the subsolution $u\in\rmV^0_{\mathrm{kin}}(\Omega_T)$ which is central to our theory. 
The following lemma is a variation of those renormalization formulas for kinetic Fokker-Planck equations appeared in \cite{Zhu22, AH24}.
\begin{lemma}[Renormalization]\label{thm:renormal}
    Let $\Omega = \calV \times \calU\subset\R^{m_0}\times\R^{N-m_0}$ such that $\partial \calV$ is $\mathrm{C}^{0,1}$ and $\partial \calU$ is $\mathrm{C}^{1,1}$. 
    If $u\in\rmV^0_{\mathrm{kin}}(\Omega_T)$ is a subsolution to \cref{eq:kol},
    then for any $\Phi:\R\to\R$ convex non-decreasing and $v\in \mathrm{C}^1(\overline{\Omega_T})$ satisfying $\Phi'\in\rmW^{1,\infty}(\R)$, $\Phi'(u(x,t)) = 0$ on $(\partial \calV) \times \calU \times (0,T)$, and $v\geq 0$, we have 
    \begin{multline}\label{eq:renormal}
        \iint_{\Omega_T} \Big(-\Phi(u) v_t + \Phi(u) \pair*{\bfB x, Dv} 
        +   a^{ij}D_j \Phi(u) D_i v 
        +    b^i u \Phi'(u) D_i v - c^i (D_i \Phi(u))v\\
        - du\Phi'(u) v 
        +  (a^{ij} D_j u + b^i u + f^i ) D_i u \Phi''(u) v 
        - g\Phi'(u) v +   f^i\Phi'(u) D_i v \Big)\,\mathrm{d}x\mathrm{d}t\\
        \leq  
        \int_{ \Gamma_{\mathrm{K}}} \langle \bfB x, \mathbf{n}_x\rangle \Phi(u_\Gamma) v\,\mathrm{d}S\mathrm{d}t
        + \left[\int_{\Omega} \Phi(u) v\,\mathrm{d}x\right]^0_T ,
    \end{multline}
    where $\rmd S$ is the $N-1$ dimensional Hausdorff measure on $\calV\times\partial\calU$.
\end{lemma}
\begin{remark}
    Regarding the existence of weak solutions in bounded domains, it has been treated for kinetic Fokker-Planck equations using the vanishing viscosity method in \cite{Zhu22,AH24} and the method can be generalized to Kolmogorov equations of the form \cref{eq:kol}.  
\end{remark}
The next lemma is a consequence of the convolution-translation technique which is utilized to obtain the weak trace and also the renormalization formula. For details, we refer to \cite{AH24}.
\begin{lemma}\label{thm:bdy}
    Let $\Omega = \calV \times \calU\subset\R^{m_0}\times\R^{N-m_0}$ such that $\partial \calV$ is $\mathrm{C}^{0,1}$ and $\partial \calU$ is $\mathrm{C}^{1,1}$. 
    If $u\in\rmV^0_{\mathrm{kin}}(\Omega_T)$ is bounded, then 
    \begin{equation*}
        \sup_{\Gamma_{\mathrm{K}}}\mathrm{tr}_{\Gamma_\mathrm{K}}(u) \leq \sup_{\Omega_T} u.
    \end{equation*}
\end{lemma}

\section{Proof of Theorem 1: the De Giorgi iteration}\label{sec:proof}
In this section, we employ the De Giorgi iteration and prove \cref{thm:1}. 
We will focus on \cref{eq:data1} first, as the proof for \cref{eq:data2} follows similarly. 
Note that besides the notations already introduced in the introduction, we denote by $a\lesssim_{\gamma} b$ to mean $a\leq C(\gamma) b$.

Before proceeding, we need the following iteration lemma, which lies at the core of the De Giorgi iteration. 
This lemma can be found in \cite{LSU68,DiB93}, among others, and the proof is an easy induction.
\begin{lemma}\label{thm:it1}
    Let $Y_n$, $n=0,1,2,\dots$, be a sequence of positive numbers satisfying the recursive inequalities 
    \begin{equation}\label{eq:thm:it1:1}
        Y_{n+1} \leq C b^n Y_n^{1+\alpha},
    \end{equation}
    where $C>0$, $b >1$ and $\alpha>0$ are given constants. 
    If 
    \begin{equation}\label{eq:thm:it1:2}
        C Y_0^\alpha \gamma \leq 1
    \end{equation}
    where $\gamma$ is such that $\gamma^\alpha = b$,
    then $Y_n \to 0$ as $n\to\infty$.
\end{lemma}

Let $0<t_1\leq T$ be a given time and $k\geq M$ a given number. We define 
\begin{equation}\label{eq:uk}
    u_k := \begin{cases}
        (u-k)_+ & \textrm{ in } \Omega_{t_1},\\
        0 & \textrm{ in } \R^{N+1}\setminus \Omega_{t_1}
    \end{cases}
\end{equation}
and also 
\begin{equation}\label{eq:ak}
    \calA_k := \{(x,t)\in\Omega_{t_1}: u(x,t) > k\}. 
\end{equation}
We shall derive iterated inequalities related to $u_k$ for appropriately chosen sequences of $k$'s. 

\subsection{$\rmL^{2p_0}$ embedding for truncations of the undercut}\label{sec:truncate}
First, we need to prove an embedding result for the subsolution similar to the uniformly parabolic case. 
However, here we can only do this for truncations of the undercut of the subsolution.

Consider the following functions $\Psi_{k,l}$ for $l>k$ where $k\geq M$ fixed:
\begin{equation}\label{eq:psi}
    \Psi_{k,l}(r) := \begin{cases}
        0 & \textrm{ for } r\leq k,\\
        2(r-k) & \textrm{ for } k<r<l,\\
        2(l-k) & \textrm{ for } r\geq l. 
    \end{cases}
\end{equation}
We define 
\begin{equation}\label{eq:phi}
    \Phi_{k,l} (r) = \int_{-\infty}^r \Psi_{k,l} (s)\,\rmd s.
\end{equation}
It is then clear by definition that $\Phi'=\Psi\in\rmW^{1,\infty}(\R)$.
Indeed,
\begin{equation}\label{eq:phi1}
    \Phi_{k,l} (r) =
    \begin{cases}
        0, & r\leq k,\\
        (r-k)^2, & k<r<l,\\
        (l-k)(2(r-k)-(l-k)), & l\leq r.
    \end{cases}
\end{equation}
By letting $l\to\infty$, we see $\Phi_{k,l}(r) \to (r-k)_+^2$ and $\Psi_{k,l}\to 2(r-k)_+$. 
We call $\Psi_{k,l}(u)$ and $\Phi_{k,l}(u)$ the truncations of the undercut $u_k$ and $u_k^2$, respectively.

The next proposition is a counterpart of the parabolic embedding for Kolmogorov equations, {  where we first establish higher integrability of $\Phi_{k,l}(u)$ that depends on the truncation parameters $k,l$. In the next section, we improve the integrability in \cref{thm:l2p0embedding1} so that it is independent of $l$.} 
\begin{proposition}\label{thm:l2p0embedding}
    Let \cref{eq:data1} hold, and $0<t_1\leq T$ be any given time. 
    If $u$ is a subsolution to \cref{eq:kol}, and $u_k$ is its undercut as defined in \cref{eq:uk}, 
    then for fixed $k$ and $l$, and $\widehat{p}_0 = p_0 -\varepsilon_0$ where $0<\varepsilon_0\leq p_0 -1$, there exists a constant $C = C(\varepsilon_0, l, \Lambda, \ref{eq:data1})$ such that
    \begin{equation}\label{eq:lpembedding}
        \norm{\Phi_{k,l}(u)}_{\widehat{p}_0, \Omega_{t_1}} < C  \left(\norm{u}_{2,\Omega_{t_1}} + \norm{\nabla_{m_0} u_k}_{2,\Omega_{t_1}}\right) {  < \infty},
    \end{equation}
    where $\nabla_{m_0}:= (D_1,\dots,D_{m_0})$ is the gradient of the first $m_0$ coordinates.
\end{proposition}
\begin{remark}
    Note that the scaling here is only correct for the tail of $\Phi_{k,l}(u)$, i.e.~the linear part.
\end{remark}
\begin{lemma}\label{thm:psi2phi}
    Assume the conditions in \cref{thm:l2p0embedding}. Then
    \begin{equation}
        \norm{\Psi_{k,l}(u)}_{2\widehat{p}_0, \Omega_{t_1}}^2 = \norm{\Psi_{k,l}^2(u)}_{\widehat{p}_0, \Omega_{t_1}} \leq 4\norm{\Phi_{k,l}(u)}_{\widehat{p}_0, \Omega_{t_1}}.
    \end{equation}
\end{lemma}
\begin{proof}
    Note that $\Psi_{k,l}^2(r) \leq 4 \Phi_{k,l}(r)$.
    To see this, recall that $\Phi_{k,l}(r)$ is given by \cref{eq:phi1}, and 
    \begin{equation*}
        \Psi_{k,l}^2(r) =
        \begin{cases}
            0, &r\leq k,\\
            4(r-k)^2, &k<r<l,\\
            4(l-k)^2, &l\leq r.
        \end{cases}
    \end{equation*}
    Observe that if $r\geq l$, $(l-k)(2(r-k)-(l-k))\geq (l-k)^2$. 
    Therefore, the claim is proved, and consequently, the lemma follows. 
\end{proof}

\begin{delayedproof}{thm:l2p0embedding}
    By taking $\Phi = \Phi_{k,l}$, noting $\Phi''_{k,l}=\Psi'_{k,l}\geq 0$ and using the ellipticity assumption \cref{assump:elliptic}, it follows from \cref{thm:renormal} that for all $v\in \mathrm{C}^1(\overline{\Omega_{t_1}})$ with $v\geq 0$, we have 
    \begin{multline*}
        \iint_{\Omega_{t_1}} -\Phi_{k,l}(u) D_t v + \Phi_{k,l}(u) \langle \bfB x, Dv\rangle + D_i \Phi_{k,l}(u) D_i v \,\rmd x\rmd t \\
        \leq \iint_{\Omega_{t_1}} \big( D_i \Phi_{k,l} (u)D_i v -  a^{ij}D_j \Phi_{k,l}(u) D_i v 
        - b^i u \Phi_{k,l}'(u) D_i v 
        + c^i (D_i \Phi_{k,l}(u))v\\
        + du\Phi_{k,l}'(u) v 
        - ( b^i u + f^i ) D_i u \Phi_{k,l}''(u) v 
        + g\Phi_{k,l}'(u) v -  f^i\Phi_{k,l}'(u) D_i v\big) \,\mathrm{d}x\mathrm{d}t.
    \end{multline*}
    Here the boundary term disappears since $k\geq M$.

    { Let's extend $u$ to be 0 on $\R^{N+1}\setminus\Omega_{t_1}$. Hence, $\Phi_{k,l}(u)$ and $\Psi_{k,l}(u)$ are also 0 on $\R^{N+1}\setminus\Omega_{t_1}$.} 
    By the Riesz representation theorem {  and the above equation}, 
    there exists a non-negative Radon measure $\mu$ such that in the sense of distributions, 
    \begin{equation}\label{eq:it1:K0uk2}
        \K_0 (\Phi_{k,l}(u)) =   D_i F^i_{k,l} + G_{k,l} -\mu,
    \end{equation}
    where for fixed $i=1,\dots,m_0$, 
    \begin{equation*}
        F^i_{k,l} = \Psi_{k,l}(u) D_i u_k -a^{ij} \Psi_{k,l}(u) D_j u_k + b^i u \Psi_{k,l}(u) + f^i \Psi_{k,l}(u),
    \end{equation*}
    and 
    \begin{equation*}
        G_{k,l} =   c^i \Psi_{k,l}(u) D_i u_k + b^i u D_i u \Psi_{k,l}'(u) + f^i D_i u \Psi_{k,l}'(u) + d u \Psi_{k,l}(u) + g \Psi_{k,l}(u).
    \end{equation*}

    Note that \cref{eq:it1:K0uk2} yields  
    \begin{equation*}
        \Phi_{k,l}(u) =  (D_i K)*F^i_{k,l} + K* G_{k,l} - K*\mu \leq  (D_i K)*F^i_{k,l} + K* G_{k,l},
    \end{equation*}
    where $K$ is the fundamental solution defined in \cref{eq:kernel}.
    Hence, by applying \cref{thm:embedding:l1}, we get for $\widehat{p}_0 = p_0 -\varepsilon_0$ that 
    \begin{equation}\label{eq:it1:embedding}
        \norm{\Phi_{k,l}(u)}_{\widehat{p}_0,\, \R^{N}\times(0,t_1)} \leq   \sigma_1 \norm{F^i_{k,l}}_{\widehat{p}_1,\, \R^{N}\times(0,t_1)} + \sigma_0 \norm{G_{k,l}}_{1,\, \R^{N}\times(0,t_1)},
    \end{equation}
    where $\widehat{p}_1$ satisfies for some $\varepsilon_1>0$ small to be determined that 
    \begin{equation}\label{eq:it1:hatp1:1}
        \frac{1}{\widehat{p}_0} = \frac{1}{\widehat{p}_1} + \frac{1}{p_1 -\varepsilon_1} -1.
    \end{equation}
    In addition, 
    we impose the following condition on $\widehat{p}_1$ (this will give us the desired H\"older conjugate later):
    \begin{equation}\label{eq:it1:hatp1:2}
        \widehat{p}_1\frac{2}{2-\widehat{p}_1} = 2\widehat{p}_0. 
    \end{equation}
    It is easy to see that as long as we choose $\widehat{p}_1= p_1 -\varepsilon_1$, then for each fixed $\varepsilon_0>0$ small,
    there exists an $\varepsilon_1>0$ small such that \cref{eq:it1:hatp1:1,eq:it1:hatp1:2} hold and we have $1<\widehat{p}_1<p_1<2$.

    Now it remains to check that the right-hand side of \cref{eq:it1:embedding} is finite.
    To do this, we inspect the terms of $F^i_{k,l}$ and $G_{k,l}$ in detail. 
    Note that all the integrals are supported on $\calA_k\subset \Omega_{t_1}$, 
    and that 
    by construction $\Psi_{k,l}(u)$ and $u\Psi_{k,l}'(u)$ are bounded by $2l$. 
    Therefore, for each $i=1,\dots,m_0$, {  by straightforward calculations and a careful application of H\"older's inequality we compute} 
    \begin{align*}
        \norm{\Psi_{k,l}(u) D_i u_k}_{\widehat{p}_1,\calA_k} \lesssim_{l} \norm{D_i u_k}_{\widehat{p}_1,\calA_{k}} \lesssim_{l} \abs{\calA_k}^{\frac{1}{2\widehat{p}_0}} \norm{D_i u_k}_{2, \calA_k},\\
        \norm{a^{ij}\Psi_{k,l}(u) D_j u_k}_{\widehat{p}_1,\calA_k} \lesssim_{l,\Lambda} \norm{D_j u_k}_{\widehat{p}_1,\calA_{k}} \lesssim_{l,\Lambda} \abs{\calA_k}^{\frac{1}{2\widehat{p}_0}} \norm{D_j u_k}_{2, \calA_k},\\
        \norm{b^i u \Psi_{k,l}(u)}_{\widehat{p}_1,\calA_{k}} \lesssim_{l} \norm{b^i u}_{\widehat{p}_1,\calA_k} \lesssim_{l} \norm{b^i}_{2\widehat{p}_0,\calA_{k}} \norm{u}_{2, \calA_{k}},\\
        \norm{f^i \Psi_{k,l}(u)}_{\widehat{p}_1, \calA_{k} } \leq \norm{f^i}_{2\widehat{p}_0,\calA_k} \norm{\Psi_{k,l}(u)}_{2,\calA_k},\\
        \norm{c^i \Psi_{k,l}(u) D_i u_k}_{1,\calA_{k} } \lesssim_{l} \norm{c^i D_i u_k}_{1,\calA_k} \lesssim_{l} \norm{c^i}_{2, \calA_k} \norm{D_i u_k}_{2,\calA_k},\\
        \norm{b^i u \Psi_{k,l}'(u) D_i u_k }_{1,\calA_{k} } \lesssim_{l} \norm{b^i  D_i u_k }_{1,\calA_{k} } \lesssim_{l} \norm{b^i}_{2,\calA_k} \norm{D_i u_k}_{2,\calA_k},\\
        \norm{f^i \Psi_{k,l}'(u) D_i u_k }_{1, \calA_{k} } \leq \norm{f^i  D_i u_k }_{1, \calA_{k} } \leq \norm{f^i}_{2,\calA_k} \norm{D_i u_k}_{2,\calA_k},\\
        \norm{d u \Psi_{k,l}(u)}_{1,\calA_{k} } \lesssim_{l} \norm{d u }_{1,\calA_{k} } \lesssim_{l} \norm{d}_{2,\calA_k} \norm{u}_{2,\calA_k},\\
        \norm{g \Psi_{k,l}(u)}_{1,\calA_k} \leq \norm{g}_{2,\calA_k} \norm{\Psi_{k,l}(u)}_{2,\calA_k}.
    \end{align*} 
    Note when $N\geq2$, $\widehat{p}_0<\tilde{q}$ always holds true.
    It follows that all the terms above are bounded by \cref{eq:data1}, hence \cref{eq:lpembedding} is proved.
\end{delayedproof}

\subsection{$\rmL^{2p_0}$ energy estimate for the undercut in sufficiently small time cylinder}\label{sec:energy}
Now we derive the $\rmL^{2\widehat{p}_0}$ embedding for $u_k$ in $\Omega_{t_1}$ for sufficiently small $t_1>0$, {  which is an improvement of \cref{thm:l2p0embedding}}.

\begin{proposition}\label{thm:l2p0embedding1}
    Let \cref{eq:data1} hold, and $0<t_1\leq T$ be a time sufficiently small (depending on \ref{eq:data1} as constructed in the proof). 
    If $u$ is a subsolution to \cref{eq:kol}, and $u_k$ is its undercut defined in \cref{eq:uk}, 
    then for fixed $k$ and $\widehat{p}_0 = p_0 -\varepsilon_0$ where $\varepsilon_0>0$ is a small constant depending on \cref{eq:data1} (constructed in the proof), the following holds:
    \begin{equation*}
        \norm{u_k^2}_{\widehat{p}_0,\calA_k} 
        \leq \gamma_2 \left(\norm{\nabla_{m_0} u_k}_{2,\calA_k}^2 
        + k^2 \norm{b^i}_{2,\calA_k}^2 + \norm{f^i}_{2,\calA_k}^2 + k^2\norm{d}_{\frac{2\widehat{p}_0}{2\widehat{p}_0-1},\calA_k}^2 + \norm{g}_{\frac{2\widehat{p}_0}{2\widehat{p}_0-1}, \calA_k}^2 \right)
    \end{equation*}
    where $\gamma_2 = \gamma_2(m_0,\Lambda, \sigma_0,\sigma_1, \ref{eq:data1})$.
\end{proposition}
\begin{proof}
    Recall the equation \cref{eq:it1:embedding}. 
    We improve the estimates of its terms based on \cref{thm:l2p0embedding,thm:psi2phi}. 
    {  In the following, we will perform delicate calculations using H\"older's and Young's inequality with carefully chosen exponents, which is standard in regularity theory.} 

    Using the assumption \cref{assump:elliptic} and \cref{eq:it1:hatp1:2}, it follows that for fixed $i,j=1,\dots,m_0$,
    \begin{align}
        \label{eq:it1:1}\norm{\Psi_{k,l}(u) D_i u_k}_{\widehat{p}_1, \calA_k}  
        \leq \frac{1}{16 m_0 \sigma_1 } \norm{\Phi_{k,l}(u)}^{}_{\widehat{p}_0, \calA_k} +  16 m_0 \sigma_1 \norm{D_j u_k}_{2,\calA_k}^2,\\
        \label{eq:it1:aij}\norm{a^{ij} \Psi_{k,l}(u) D_j u_k}_{\widehat{p}_1, \calA_{k} } 
        \leq \frac{1}{16 m_0^2 \sigma_1} \norm{\Phi_{k,l}(u)}_{\widehat{p}_0, \calA_k} + 16 m_0^2 \Lambda^2 \sigma_1 \norm{D_j u_k}_{2,\calA_k}^2 ;
    \end{align}
    Next, observe 
    \begin{equation}\label{eq:upsi}
        u\Psi_{k,l}(u) \leq 2\Phi_{k,l}(u) + k \Psi_{k,l}(u).
    \end{equation}
    Indeed, 
    \begin{equation*}
        u\Psi_{k,l}(u) = 
        \begin{cases}
            0, & u\leq k,\\
            2(u-k)u, & k< u <l,\\
            2(l-k)u, &l\leq u, 
        \end{cases}
    \end{equation*}
    and by writing $u= (u-k) + k$, we get 
    \begin{equation*}
        u\Psi_{k,l}(u) = 
        \begin{cases}
            0, & u\leq k,\\
            2(u-k)((u-k)+k) = 2 (u-k)^2 + k 2(u-k), & k< u <l,\\
            2(l-k)((u-k)+k) = 2(u-k)(l-k) + k 2(l-k), &l\leq u. 
        \end{cases}
    \end{equation*}
    Note if $l\leq r$, then $(l-k)(2(r-k)-(l-k))\geq (r-k)(l-k)$, thus \cref{eq:upsi} is proved. 
    Then, for fixed $i=1,\dots,m_0$,
    \begin{multline}\label{eq:it1:biupsi}
        \norm{b^i u \Psi_{k,l}(u)}_{\widehat{p}_1, \calA_k} \leq 2\norm{b^i \Phi_{k,l}(u)}_{\widehat{p}_1,\calA_k} + k\norm{b^i\Psi_{k,l}(u)}_{\widehat{p}_1,\calA_k}  \\
        \leq 2\norm{b^i}_{\frac{\widehat{p}_0 \widehat{p}_1}{\widehat{p}_0-\widehat{p}_1},\calA_k}\norm{\Phi_{k,l}(u)}_{\widehat{p}_0, \calA_{k}} + 16 m_0 \sigma_1 k^2 \norm{b^i}_{2, \calA_{k}}^2+ \frac{1}{16 m_0\sigma_1}\norm{\Phi_{k,l}(u)}_{\widehat{p}_0,\calA_k}
    \end{multline}
    Next, for fixed $i=1,\dots,m_0$,
    \begin{equation}\label{eq:it1:fipsi}
        \norm{f^i \Psi_{k,l}(u)}_{\widehat{p}_1,\calA_k} 
        \leq 16 m_0\norm{f^i}_{2,\calA_k}^2 + \frac{1}{16m_0}\norm{\Phi_{k,l}(u)}_{\widehat{p}_0,\calA_k};
    \end{equation}
    and
    \begin{equation}\label{eq:it1:cipsi}
        \norm{c^i \Psi_{k,l}(u) D_i u_k}_{1,\calA_k} 
        \leq \frac{1}{16 m_0 \sigma_0 } \norm{\Phi_{k,l}(u)}_{\widehat{p}_0, \calA_k} + 16 m_0 \sigma_0 \norm{c^i }_{\frac{2\widehat{p}_0}{\widehat{p}_0 - 1}, \calA_k}^2 \norm{D_i u_k}_{2,\calA_k}^2;
    \end{equation}
    Next, note that $u\Psi_{k,l}'(u) \leq \Psi_{k,l}(u) + 2k$, then for fixed $i=1,\dots,m_0$,
    \begin{multline}\label{eq:it1:biupsi'}
            \norm{b^i u \Psi_{k,l}'(u) D_i u_k}_{1,\calA_k} \leq \norm{b^i \Psi_{k,l}(u)D_i u_k}_{1,\calA_k} + 2k\norm{b^i D_i u_k}_{1,\calA_k} \\
            \leq \frac{1}{16 m_0 \sigma_0}\norm{\Phi_{k,l}(u) }_{\widehat{p}_0, \calA_k} + k^2\norm{b^i}_{2,\calA_k}^2  + \left(16 m_0\sigma_0 \norm{b^i}_{\frac{2\widehat{p}_0}{\widehat{p}_0 - 1}, \calA_k}^2 + 1 \right)\norm{D_i u_k}_{2,\calA_k}^2;
    \end{multline}
    Next, for fixed $i=1,\dots,m_0$,
    \begin{equation}\label{eq:it1:fipsi'}
        \norm{f^i \Psi_{k,l}'(u)D_i u_k}_{1,\calA_k} \leq \norm{f^i}_{2,\calA_k}^2 + \norm{D_i u_k}_{2,\calA_k}^2;
    \end{equation}
    and 
    \begin{multline}\label{eq:it1:dupsi}
        \norm{du\Psi_{k,l}(u)}_{1,\calA_k} \leq 2\norm{ d \Phi_{k,l}(u)}_{1,\calA_k} + k\norm{d\Psi_{k,l}(u)}_{1,\calA_{k}}\\
        \leq 2\norm{d}_{\frac{\widehat{p}_0}{\widehat{p}_0 - 1},\calA_k}\norm{\Phi_{k,l}(u)}_{\widehat{p}_0,\calA_k} + 16\sigma_0 k^2\norm{d}_{\frac{2\widehat{p}_0}{2\widehat{p}_0-1},\calA_k}^2 + \frac{1}{16\sigma_0}\norm{\Phi_{k,l}(u)}_{\widehat{p}_0,\calA_k};
    \end{multline}
    and
    \begin{equation}\label{eq:it1:gpsi}
        \norm{g\Psi_{k,l}(u)}_{1,\calA_k} 
        \leq 16\sigma_0 \norm{g}_{\frac{2\widehat{p}_0}{2\widehat{p}_0-1}, \calA_k }^2 +  \frac{1}{16\sigma_0}\norm{\Phi_{k,l}(u)}_{\widehat{p}_0 ,\calA_k}.
    \end{equation}
    Inspecting the exponents for $b^i,c,d,f^i,g$ in \cref{eq:it1:biupsi,eq:it1:fipsi,eq:it1:cipsi,eq:it1:biupsi',eq:it1:fipsi',eq:it1:dupsi,eq:it1:gpsi}, we can choose $\varepsilon_0$ sufficiently small according to \cref{eq:data1} so that all the terms above containing coefficients are finite.
    Combining all the estimates \cref{eq:it1:1,eq:it1:aij,eq:it1:biupsi,eq:it1:fipsi,eq:it1:cipsi,eq:it1:biupsi',eq:it1:fipsi',eq:it1:dupsi,eq:it1:gpsi}, then \cref{eq:it1:embedding} becomes 
    \begin{multline}
        \norm{\Phi_{k,l}(u)}_{\widehat{p}_0, \calA_k} \leq \gamma_1 \norm{\Phi_{k,l}(u)}_{\widehat{p}_0,\calA_k} \\
        + \gamma_2 \left(\norm{\nabla_{m_0} u_k}_{2,\calA_k}^2 
        + k^2 \norm{b^i}_{2,\calA_k}^2 + \norm{f^i}_{2,\calA_k}^2 + k^2\norm{d}_{\frac{2\widehat{p}_0}{2\widehat{p}_0-1},\calA_k}^2 + \norm{g}_{\frac{2\widehat{p}_0}{2\widehat{p}_0-1}, \calA_k}^2 \right),
    \end{multline}
    where
    \begin{gather*}
        \gamma_1 = \frac{1}{2} + 2m_0\sigma_1 \norm{b^i}_{\frac{\widehat{p}_0 \widehat{p}_1}{\widehat{p}_0 - \widehat{p}_1},\Omega_{t_1}} + 2\sigma_0 \norm{d}_{\frac{\widehat{p}_0}{\widehat{p}_0 -1},\Omega_{t_1}},\\
        \gamma_2 = \gamma_2 \left(m_0, \Lambda, \sigma_0, \sigma_1, \ref{eq:data1}\right).
    \end{gather*}
    It is clear that we can then choose $t_1$ so small that $\gamma_1\leq \frac{3}{4}$, thus
    \begin{multline*}
        \frac{1}{4}\norm{\Phi_{k,l}(u)}_{\widehat{p}_0,\calA_k} \\
        \leq \gamma_2 \left(\norm{\nabla_{m_0} u_k}_{2,\calA_k}^2 
        + k^2 \norm{b^i}_{2,\calA_k}^2 + \norm{f^i}_{2,\calA_k}^2 + k^2\norm{d}_{\frac{2\widehat{p}_0}{2\widehat{p}_0-1},\calA_k}^2 + \norm{g}_{\frac{2\widehat{p}_0}{2\widehat{p}_0-1}, \calA_k}^2 \right).
    \end{multline*}
    Since the right-hand side is independent of $l$, letting $l\to\infty$ and using the monotone convergence theorem completes the proof.  
\end{proof}

\subsection{Caccioppoli estimate for the undercut in sufficiently small time cylinder}\label{sec:caccioppoli}
From this point onward, we can establish a Caccioppoli estimate for the undercut $u_k$.
\begin{proposition}\label{thm:it1:caccioppoli}
    Let $u_k$ be defined as in \cref{eq:uk}. For $t_1>0$ sufficiently small (as determined in \cref{thm:l2p0embedding1}), we have
    \begin{equation}\label{eq:it1:caccioppoli}
        \iint_{\calA_k} \abs{\nabla_{m_0} u_k}^2 \,\rmd z \lesssim_{\lambda} \iint_{\calA_k} \abs{c^i  u_k}^2 + \abs{d u_k^2} + k\abs{d u_k} 
        + \abs{b^i (u_k+k)}^2 + \abs{f^i}^2  + \abs{g u_k}\,\rmd z.
    \end{equation}
\end{proposition}
\begin{proof}
    Choosing $\Phi = \Phi_{k,l}$ and $v=1$ in \cref{thm:renormal}, we find that 
    \begin{multline*}
        \iint_{\Omega_{t_1}} -c^i \Psi_{k,l}(u) D_i u_k - du \Psi_{k,l}(u) + a^{ij}D_j u_k D_i u_k \Psi_{k,l}'(u) + b^i u D_i u_k \Psi_{k,l}'(u) \\
        + f^i D_i u_k \Psi_{k,l}'(u) - g \Psi_{k,l}(u) \,\rmd x\rmd t 
        \leq 0
    \end{multline*}
    { Therefore, using the elliptic assumption \cref{assump:elliptic} and the estimates \cref{eq:upsi,eq:it1:biupsi'} from the proof of \cref{thm:l2p0embedding1}, we obtain 
    \begin{multline*}
        \int_{\calA_k} \lambda \abs{\nabla_{m_0} u_k}^2 \Psi_{k,l}'(u)\,\rmd z \leq \int_{\calA_k} \frac{1}{\lambda}\abs{c^i \Psi_{k,l}(u)}^2 + \frac{\lambda}{4}\abs{D_i u_k}^2 + \abs{d(2\Phi_{k,l}(u)+k\Psi_{k,l}(u))} \\
        + \frac{1}{\lambda}\abs{b^i(\Psi_{k,l}(u)+2k)}^2 + \frac{\lambda}{4}\abs{D_i u_k}^2 + \frac{1}{\lambda}\abs{\Psi_{k,l}'(u)f^i}^2 + \frac{\lambda}{4}\abs{D_i u_k}^2 + \abs{g\Psi_{k,l}(u)}\,\rmd x\rmd t.  
    \end{multline*}
    From \cref{thm:l2p0embedding1}, we know $\norm{\Psi_{k,l}(u)}_{2\widehat{p}_0,\calA_k}^2$ is uniformly bounded in $l$, hence we first let $l\to\infty$ and then absorb the terms $\abs{\nabla_{m_0}u_k}^2$ to derive \cref{eq:it1:caccioppoli}.}
\end{proof}

\subsection{Iteration inequalities}\label{sec:iteration}
{ Now, by combining \cref{thm:l2p0embedding1,thm:it1:caccioppoli}, and repeating the arguments similar to the proof of \cref{thm:l2p0embedding1}}, we can choose $t_1$ to be even smaller, depending on $\gamma_2$ and $\lambda$, to obtain  
\begin{equation*}
    \norm{u_k^2}_{\widehat{p}_0,\calA_k} \\
    \leq \gamma \left( 
    k^2 \norm{b^i}_{2,\calA_k}^2 + \norm{f^i}_{2,\calA_k}^2 + k^2\norm{d}_{\frac{2\widehat{p}_0}{2\widehat{p}_0-1},\calA_k}^2 + \norm{g}_{\frac{2\widehat{p}_0}{2\widehat{p}_0-1}, \calA_k}^2 \right).
\end{equation*}

{ Next, choose 
$q'<\widehat{p}_0$ to be determined later and 
\begin{equation*}
    q'' = \frac{2\widehat{p}_0 2q'}{(2\widehat{p}_0-1)2q' - 2\widehat{p}_0},
\end{equation*}
it follows from H\"older's inequality that }
\begin{align*}
    \norm{b^i}_{2,\calA_k}^2 \leq \norm{b^i}_{\frac{2q'}{q'-1},\calA_k}^2 \abs*{\calA_k}^{\frac{1}{q'}},&
    \qquad \norm{f^i}_{2,\calA_k}^2 \leq \norm{f^i}_{\frac{2q'}{q'-1},\calA_k}^2 \abs*{\calA_k}^{\frac{1}{q'}},\\
    \norm{d}_{\frac{2\widehat{p}_0}{2\widehat{p}_0 - 1}, \calA_k}^2 \leq \norm{d}_{q'', \calA_k}^2 \abs*{\calA_k}^{\frac{1}{q'}},&
    \qquad \norm{g}_{\frac{2\widehat{p}_0}{2\widehat{p}_0 - 1}, \calA_k}^2 \leq \norm{g}_{q'', \calA_k}^2 \abs*{\calA_k}^{\frac{1}{q'}}.
\end{align*}

It can be verified that {  $2q'/(q'-1)$ and $q''$ decrease to the exponents specified in \cref{eq:data1} as $q'$ approaches $\widehat{p}_0$. Therefore, they are admissible exponents for the coefficients, provided $q'$ is sufficiently close to $\widehat{p}_0$}. 

Combining all the estimates above, and assuming $k\geq 1$, we get 
\begin{equation}\label{eq:it1:iteration}
    \norm{u_k}_{2\widehat{p}_0 , \calA_k}^2 \leq \gamma k^2 \abs*{\calA_k}^{\frac{1}{q'}}
\end{equation}
where, abusing notation slightly, $\gamma$ now depends additionally on $\norm{b^i}_{\frac{2q'}{q'-1},\calA_k}^2$, $\norm{f^i}_{\frac{2q'}{q'-1},\calA_k}^2 $, $\norm{d}_{q'', \calA_k}^2$, and $\norm{g}_{q'', \calA_k}^2$.

\subsection{Proof of \cref{eq:linfl2}}\label{sec:proof1}
To prove \cref{eq:linfl2}, we define the following quantities: let 
\begin{equation*}
    k_n := M+ k -\frac{k}{2^n}
\end{equation*}
for some $k\geq \max\{1,M\}$, and denote $\calA_n :=\calA_{k_n} $ {  and $u_n:= u_{k_n}$}. 
Define 
\begin{equation*}
    Y_n:=\int_{\Omega_{t_1}} \abs{u_n}^2\,\rmd z.
\end{equation*}
It follows {  from chebyshev's inequality that 
\begin{equation*}
    \abs{\calA_{n+1}} \leq  \frac{2^{2(n+1)}}{k^2} Y_n .
\end{equation*}}
Thus, by { using \cref{eq:it1:iteration} and} setting $\alpha = \frac{1}{q'}-\frac{1}{\widehat{p}_0}$, we derive 
\begin{equation*}
    Y_{n+1} \leq \gamma k^2 \left\vert \frac{2^{2n+2}}{k^2} Y_n\right\vert^{1 + \alpha}  \leq  \frac{\gamma}{k^{2\alpha}} 2^{2(1+\alpha)n} Y_n^{1+ \alpha}.
\end{equation*}
This is of the form \cref{eq:thm:it1:1} {  in \cref{thm:it1}}, and by choosing 
\begin{equation*}
    \frac{\gamma}{k^{2\alpha}} Y_0^\alpha 2^{\frac{2}{\alpha}+ 2}\leq 1,
\end{equation*}
which implies 
\begin{equation*}
    k^{2\alpha}\geq \gamma 2^{\frac{2}{\alpha} + 2} \norm{(u-M)_+}_{2,\Omega_{t_1}}^{2\alpha},
\end{equation*}
we conclude $\lim_{n\to\infty } Y_n = 0$.

Therefore,  
\begin{equation*}
    \sup_{\Omega_{t_1}} u \leq M + C_0 \max\left\{1, M, \norm{(u-M)_+}_{2,\Omega_{t_1}}\right\}
\end{equation*}
where $C_0 = C_0\mathrm{\cref{eq:data1}}$. 

By propagating this inequality over finitely many subintervals of $(0,T)$, we obtain the desired inequality \cref{eq:linfl2} {  in \cref{thm:1}}.
For instance, on $(t_1,t_2)$, we may choose 
\begin{equation*}
    M_1 = M + C_0 \max\{1, M, \norm{(u-M)_+}_{2,\Omega_{t_1}}\}
\end{equation*}
and by the above iteration procedure, we get 
\begin{equation*}
    \sup_{\Omega_{t_1,t_2}} u \leq M_1 + C_1 \max\left\{1, M, \norm{(u-M)_+}_{2,\Omega_{t_1,t_2}}\right\}.
\end{equation*} 
Combining all such estimates yields the desired inequality \cref{eq:linfl2} where the boundary inequality is a consequence of \cref{thm:bdy}.

\subsection{Proof of \cref{eq:linfdata}}\label{sec:proof2}
Next, we set 
\begin{equation*}
    k_n := h\left(2-\frac{1}{2^n}\right), \quad n=0,1,2,\dots
\end{equation*}
and define $u_n :=u_{k_n}$ and $\calA_n :=\calA_{k_n} $.

Observe that {  from Chebyshev's inequality we have}
\begin{equation*}
    \norm{u_n}_{2\widehat{p}_0, \calA_n} \geq (k_{n+1} - k_n) \abs{\calA_{n+1}}^{\frac{1}{2\widehat{p}_0}},
\end{equation*}
thus 
\begin{equation}\label{eq:it:3}
    \abs{\calA_{n+1}}^{\frac{1}{2\widehat{p}_0}} \leq \sqrt{\gamma} \frac{k_n}{k_{n+1} - k_n} \abs{\calA_{n}}^{\frac{1}{2q'}} \leq \sqrt{\gamma} 2^n \abs{\calA_{n}}^{\frac{1}{2q'}}
\end{equation}
which simplifies to
\begin{equation*}
    \abs{\calA_{n+1}} \leq \gamma^{\widehat{p}_0} (2^{2\widehat{p}_0})^n \abs{\calA_{n}}^{\frac{\widehat{p}_0}{q'}}.
\end{equation*}

Recalling that we have chosen $q'<\widehat{p}_0$, hence $\frac{\widehat{p}_0}{q'}>1$, we can apply \cref{thm:it1} provided \cref{eq:thm:it1:2} holds, i.e.
\begin{equation*}
    \gamma^{\widehat{p}_0} \abs{\calA_0}^{\frac{\widehat{p}_0}{q'}-1} \eta \leq 1
\end{equation*}
where $\eta^{\frac{\widehat{p}_0}{q'}-1} = 2^{2\widehat{p}_0}$.
To satisfy this condition, we set $h = \sigma \max\{M,1\}$ for some $\sigma>1$, replace $k_{n+1}$ by $h$, and $k_n$ by $M$ (or 1 if $M<1$) in \cref{eq:it:3} to obtain 
\begin{equation*}
    \abs{\calA_0}^{\frac{1}{2\widehat{p}_0}} \leq \sqrt{\gamma} \frac{M}{h-M} \abs{\{u>M\}}^{\frac{1}{2q'}} \leq \sqrt{\gamma} \frac{1}{\sigma - 1} \abs{\{u>M\}}^{\frac{1}{2q'}} \leq \frac{\sqrt{\gamma}}{\sigma - 1} \abs{t_1}^{\frac{1}{2q'}} \abs{\Omega}^{\frac{1}{2q'}}.
\end{equation*}
Consequently, by choosing 
\begin{equation*}
    \sigma = 1 + \gamma\abs{t_1}^{\frac{1}{2q'}}\abs{\Omega}^{\frac{1}{2q'}} \eta^{\frac{1}{2\widehat{p}_0}\frac{q'}{\widehat{p}_0 - q'}},
\end{equation*}
we obtain  
\begin{equation*}
    \sup_{\Omega_{t_1}} u \leq (1+C_0) \max\{M,1\} ,
\end{equation*}
where $C_0 = \sigma-1$. 

By propagating this on finitely many subintervals of $(0,T)$ as done in \cref{sec:proof1}, we derive the desired inequality \cref{eq:linfdata} {  in \cref{thm:1}} where the boundary inequality is a consequence of \cref{thm:bdy}.

\subsection{Iteration under Data 2}\label{sec:proof:data2}
The general framework for the proof under \cref{eq:data2} is the same as before, with the difference that in this case, we can use \cref{thm:embedding:l2} to achieve sharp exponents. Here, we outline the proof idea.

Recall for any $0<t_1 < T$ we define $u_k$ and $\calA_k$ as in \cref{eq:uk,eq:ak}.
Since $b^i,f^i=0$,
by setting $\Phi(\cdot) = (\cdot-k)_+$ in \cref{eq:renormal} (by approximating $(r-k)_+$ by convex non-decreasing $\mathrm{C}^{1,1}(\R)$ functions, e.g.~\cite[Lemma 7.6]{GT01}, and using the convexity), we get for all $v\in \mathrm{C}^1(\overline{\Omega_T})$ with $v\geq 0$ that 
\begin{multline}\label{eq:renormalsub:divfree}
    \iint_{\Omega_{t_1}} \left(-u_k D_t v + u_k \pair*{\bfB x, Dv}
    +   a^{ij}D_j u_k D_i v\right)\,\rmd x\rmd t \\
    \leq \iint_{\Omega_{t_1}} \left(  c^i (D_i u_k)v
    + du\chi_{\{u>k\}} v  
    + g\chi_{\{u>k\}} v \right)\,\mathrm{d}x\mathrm{d}t.
\end{multline}

It follows from the Riesz representation theorem that there exists a non-negative Radon measure $\mu$ such that 
\begin{equation}\label{eq:kernelrepr}
    \K_0 u_k =   D_i F^i_k + G_k -\mu,
\end{equation}
where for each $i=1,\dots,m_0$, $ F^i = D_i u_k -a^{ij} D_j u_k$,
and $G =   c^i D_i u_k + d u\chi_{\{u>k\}} + g \chi_{\{u>k\}}$.

Similar to the case of \cref{eq:data1}, there is no known Sobolev embedding theorems from the function space $\rmV^0_{\mathrm{kin}}(\Omega_T)$ itself, thus we need to show the finiteness of $\norm{u_k}_{2p_0, \calA_k}$. 
However, the presence of $du\chi_{\{u>k\}}$ in \cref{eq:renormalsub:divfree} introduces a challenge since there is not enough integrability on $u$. 
{  More precisely, we aim to apply the H\"older inequality to the term $du\chi_{\{u>k\}}$, and we intend to distribute $\rmL^{\tilde{q}}$ norm onto $d$. However, there is no any a priori higher integrability of $u\chi_{\{u>k\}}$, except for $\rmL^2$.} 
This difficulty is also noted in \cite[Sec.~3.1]{AR22}, where the authors constructed an "ad hoc" Sobolev embedding for solutions by choosing a { smaller exponent on $d$} and continuing with the Moser iteration. 
Here, we use a similar approach, but to maintain consistency with the De Giorgi iteration, we apply a bootstrapping argument to achieve higher integrability.

\begin{lemma}[Sobolev embedding]\label{thm:embedding:data2}
    Let \cref{eq:data2} holds, and $0<t_1\leq T$ be any given time. 
    If $u$ is a subsolution to \cref{eq:kol}, and $u_k$ be its undercut defined as in \cref{eq:uk}, then $\norm{u_k}_{2p_0,\calA_k}<\infty$.
\end{lemma}
\begin{proof}
    Recall from \cref{eq:data2} that $d\in\rmL^{\tilde{q}}(\Omega_T)$ for some $\tilde{q}>q_0$. 
    We define $\rho_0 = 2$, 
    and for $l=1,2,\dots,\tau$, where $\tau$ is a number to be determined, 
    \begin{equation}
        \rho_l = \frac{\rho_{l-1}\tilde{q}(Q+2)}{(\rho_{l-1}+\tilde{q})(Q+2)- 2\rho_{l-1}\tilde{q}}.
    \end{equation}
    It's easy to verify that $\rho_l$ is increasing, hence $\rho_l \geq  2$ for all $l=1,\dots,\tau$, and 
    \begin{equation*}
        \frac{\rho_l}{\rho_{l-1}} \geq \frac{1}{2\left(\frac{1}{\tilde{q}}- \frac{1}{q_0}\right)+1}.
    \end{equation*}
    Therefore, there exists a number $\tau$ such that $\rho_\tau\geq 2p_0$ and $\rho_{\tau-1}<2p_0$, and W.L.O.G.~we can assume $\rho_{\tau}=2p_0$ (by adjusting $\tilde{q}$).  

    We now show that we can bootstrap to achieve integrability up to $\rho_\tau$.
    Recall \cref{eq:kernelrepr}, it follows from \cref{thm:embedding:l2} that 
    \begin{equation}\label{eq:it2:embedding:1}
        \norm{u_k}_{\rho_l,\R^{N+1}} 
        \leq \varsigma_1 \norm{(1-a^{ij})D_j u_k}_{\omega_l, \calA_k} + \varsigma_0(\norm{c^i D_i u_k}_{\rho_l^*,\calA_k} + \norm{du}_{\rho_l^*, \calA_k} + \norm{g}_{\rho_l^*,\calA_k}),
    \end{equation}
    where $\omega_l, \rho_l^*>1$ such that
    \begin{equation*}
        \frac{1}{\rho_l} = \frac{1}{\rho_l^*} -\frac{2}{Q+2}, \textrm{ and } \frac{1}{\rho_l} = \frac{1}{\omega_l} - \frac{1}{Q+2}.
    \end{equation*}
    Noting $2<\rho_l< 2p_0$ for all $1\leq l \leq \tau-1$, it follows that 
    \begin{equation*}
        \omega_l \leq 2, \textrm{ and } \rho_l^* \leq q_1:=2\frac{Q+2}{Q+4}<2. 
    \end{equation*}
    Hence, using the monotonicity of $\rmL^p$ norm on bounded domains and \cref{eq:it2:embedding:1} yields, for a constant $C$ depends on $\varsigma_0$, $\varsigma_1$ and $\abs{\Omega_T}$ that  
    \begin{equation*}
        \norm{u_k}_{\rho_l,\calA_k} 
        \leq C (\norm{(1-a^{ij})D_j u_k}_{2, \calA_k} + \norm{c^i D_i u_k}_{q_1,\calA_k} + \norm{du}_{\rho_l^*, \calA_k} + \norm{g}_{2,\calA_k}).
    \end{equation*}
    Now, it is easy to estimate, { again using H\"older's inequality with carefully chosen exponents},
    \begin{align*}
        &\norm{c^i D_i u_k}_{q_1,\calA_k} \leq \norm{c^i}_{Q+2, \calA_k} \norm{D_i u_k}_{2,\calA_k},\\
        &\norm{du}_{\rho_l^*,\calA_k} \leq \norm{d}_{\tilde{q}, \calA_k} \norm{u}_{\rho_{l-1},\calA_k} \leq \norm{d}_{\tilde{q}, \calA_k} (\norm{u_k}_{\rho_{l-1},\calA_k} + k\abs{\calA_k}^{\frac{1}{\rho_{l-1}}}).
    \end{align*}
    Combining \cref{assump:elliptic} and the above, we get, 
    \begin{equation*}
        \norm{u_k}_{\rho_l,\calA_k} 
        \leq C (\norm{D_i u_k}_{2, \calA_k} + \norm{u_k}_{\rho_{l-1}, \calA_k} + k\abs{\calA_k}^{\frac{1}{\rho_{l-1}}} + \norm{g}_{2,\calA_k}^2),
    \end{equation*}
    where $C$ now additionally depends on $\Lambda$, $\norm{c^i}_{Q+2,\Omega_T}$ and $\norm{d}_{\tilde{q},\Omega_T}$.

    Finally, when $l=0$, $\norm{u_k}_{2,\calA_k}$ is finite, and then we can bootstrap on $\norm{u_k}_{\rho_l, \calA_k}$ to conclude $\norm{u_k}_{2p_0, \calA_k}<\infty$. 
\end{proof}
\begin{remark}
    This is not yet a Moser iteration since we cannot continue the bootstrapping for $\rho_l > 2p_0$. To proceed with the Moser iteration, we need to return to the renormalization formula \cref{eq:renormal} and choose $\Phi$ such that it approximates $\abs{\cdot}^q$, see e.g.~\cite{GM22}. 
\end{remark}
The following Caccioppoli estimate {  is} obtained from the renormalization {  \cref{thm:renormal}}, together with the truncations \cref{eq:phi,eq:psi}. 
\begin{lemma}
    Let \cref{eq:data2} hold, and let $0<t_1\leq T$ be any time. 
    If $u$ is a subsolution to \cref{eq:kol}, and $u_k$ is defined as in \cref{eq:uk}, then 
    \begin{equation}\label{eq:it2:caccioppoli}
        \int_{\calA_k} \abs{\nabla_{m_0} u_k}^2 \,\rmd z \lesssim_{\lambda} \int_{\calA_k} \abs{c^i  u_k}^2 + \abs{d u u_k} + \abs{g u_k}\,\rmd z.
    \end{equation}
\end{lemma}

Next, we derive an inequality of the form \cref{eq:it1:iteration} { which is the key iteration inequality}. 
Recall \cref{eq:kernelrepr}, 
using \cref{thm:embedding:l2} we obtain 
\begin{equation*}
    \norm{u_k}^2_{2p_0,\calA_k} \leq \varsigma_1 \norm{(1-a^{ij})D_j u_k}_{2,\calA_k}^2 
    + \varsigma_0(\norm{c^i D_i u_k}_{q_1,\calA_k}^2  
    +  \norm{d u }_{q_1,\calA_k}^2 + \norm{g}_{q_1,\calA_k}^2),
\end{equation*}
where $q_1 = 2(Q+2)/(Q+4)$. 
{  In the following, the calculations are carried out using H\"older's and Young's with carefully chosen exponents.} 
Note that 
\begin{equation*}
    \norm{c^i D_i u_k}_{q_1,\calA_k}^2 \leq \norm{c^i}^2_{Q+2, \calA_k} \norm{D_i u_k}^2_{2,\calA_k}.
\end{equation*}
Combining \cref{eq:it2:caccioppoli} and the above it follows that 
\begin{equation*}
    \norm{u_k}^2_{2p_0,\calA_k} \leq \gamma_1 (\norm{c^i u_k}_{2,\calA}^2 + \norm{duu_k}_{1,\calA_k} + \norm{gu_k}_{1,\calA_k}
    + \norm{d u }_{q_1,\calA_k}^2 + \norm{g}_{q_1,\calA_k}^2),
\end{equation*}
where $\gamma_1 = \gamma_1 (\Lambda, \norm{c^i}_{Q+2,\Omega_T}^2, \varsigma_0,\varsigma_1)$.

Next, we estimate 
\begin{equation*}
    \norm{c^i u_k}_{2,\calA_k}^2 \leq \norm{c^i}_{Q+2, \calA_k}^2 \norm{u_k}_{2p_0, \calA_k}^2;
\end{equation*}
For some $q'<p_0$ we compute 
\begin{align*}
    &\norm{du}_{q_1,\calA_k}^2 \leq \norm{d u_k}_{q_1,\calA_k}^2 + \norm{d k}_{q_1,\calA_k}^2\leq \norm{d}_{\frac{Q+2}{2},\calA_k}^2 \norm{u_k}_{2p_0, \calA_k}^2 + k^2 \norm{d}_{\frac{2q' q_1}{2q'-q_1},\Omega_{t_1}}^2 \abs{\calA_k}^{\frac{1}{q'}},\\
    &\norm{duu_k}_{1,\calA_k} \leq 2\norm{d u_k^2}_{1,\calA_k} + \norm{d k^2}_{1,\calA_k} \leq 2\norm{d}_{\frac{Q+2}{2},\calA_k} \norm{u_k}_{2p_0,\calA_k}^2 + k^2\norm{d}_{\frac{q'}{q'-1},\Omega_{t_1}}\abs{\calA_k}^{\frac{1}{q'}};
\end{align*}
We continue to estimate the terms containing $g$:
\begin{align*}
    &\norm{g}_{q_1,\calA_k}^2 \leq \norm{g}_{\frac{2q'q_1}{2q'-q_1},\Omega_{t_1}}^2 \abs{\calA_k}^{\frac{1}{q'}},\\
    &\norm{g u_k}_{1,\calA_k} \leq \norm{g}_{q'',\Omega_{t_1}}^2\abs{\calA_k}^{\frac{1}{q'}} + \frac{1}{2}\norm{u_k}_{2p_0,\calA_k}^2,
\end{align*}
where $q''>1$ satisfying $\frac{1}{q''} + \frac{1}{2p_0} + \frac{1}{2q'}=1$.
All in all, for $t_1$ sufficiently small (depending only on \cref{eq:data2}) we have that 
\begin{equation}\label{eq:it2:iteration}
    \norm{u_k}_{2p_0,\calA_k}^2 \leq \gamma k^2 \abs{\calA_k}^{\frac{1}{q'}}
\end{equation}
where $\gamma=\gamma(\lambda,\Lambda,\ref{eq:data2})$. 
The De Giorgi iteration can be carried through by following the proof strategy for \cref{eq:data1} {  in \cref{sec:proof1,sec:proof2}, and \cref{eq:linfl2,eq:linfdata} in \cref{thm:1} under \cref{eq:data1} follows easily.}

\section{Proof of Theorem 2 and some remarks}\label{sec:remarks}
\begin{delayedproof}{thm:2}
    By the assumption in \cref{thm:2}, the renormalization formula \cref{eq:renormal} can be rewritten as 
    \begin{multline*}
        \iint_{\Omega_T} \Big(-\Phi(u) v_t + \Phi(u) \pair*{\bfB x, Dv} 
        +   a^{ij}D_j \Phi(u) D_i v 
        - (b^i + c^i) (D_i \Phi(u))v\\ 
        +  a^{ij} D_j u D_i u \Phi''(u) v \Big)\,\mathrm{d}x\mathrm{d}t
        \leq  \iint_{\Omega_T} d(u\Phi'(u) v) - b^i D_i (u \Phi'(u) v)\,\rmd x\rmd t \\
        +\int_{ \Gamma_{\mathrm{K}}} \langle \bfB x, \mathbf{n}_x\rangle \Phi(u_\Gamma) v\,\mathrm{d}S\mathrm{d}t
        + \left[\int_{\Omega} \Phi(u) v\,\mathrm{d}x\right]^0_T ,
    \end{multline*}
    Formally, by choosing $\Phi(r)=(r-k)_+$ where $k\geq M$ (this can be made rigorous by approximating by $\mathrm{C}^{1,1}(\R)$ functions), we have $u\Phi'(u) v\geq 0$. It follows that  
    \begin{equation*}
        \int_{\Omega_T} -u_k v_t + u_k \pair*{\bfB x, Dv} 
        +   a^{ij}D_j u_k D_i v 
        - (b^i + c^i) (D_i u_k)v\,\rmd z \leq 0,
    \end{equation*}
    where $u_k:= (u-k)_+$.
    By inspecting the proof of \cref{thm:1}, we see by the same procedure that the right-hand side of \cref{eq:it2:iteration} will be zero, which concludes the proof.
\end{delayedproof}

We remark that the Moser iteration can also be carried out by carefully adjusting the renormalization function $\Phi$ to approximate $\abs{\cdot}^q$, as done for kinetic Fokker-Planck equations in \cite[Lemma 10]{GM22}.
Moreover, the local boundedness result can also be established by testing with suitable cutoff functions, as discussed in \cite[Prop.~11]{GM22}.

Recall \cref{assump:B}. Indeed, this assumption is rather strong, since the sufficient condition for Kolmogorov equations of the form \cref{eq:kol} to be hypoelliptic is to assume
\begin{equation*}
    \bfB = 
    \begin{pmatrix}
        * &* & \cdots & * & * \\
        \bfB_1 & * & \cdots & * & * \\
        \mathbf{O} & \bfB_2 & \cdots & * & * \\
        \vdots & \vdots & \ddots & \vdots & \vdots \\
        \mathbf{O} & \mathbf{O} & \cdots & \bfB_{\kappa} & *
    \end{pmatrix}
\end{equation*}
where $*$ denotes arbitrary elements. 
In this case, the function space is different and the structure of the fundamental solution to \cref{eq:principal} is more complicated. 
One of the main difficulties is that Young's inequality, \cite[Lemma 1.4]{BCD11}, is unknown because the Lie group structure is no longer homogeneous, so \cite[Prop. 1.3.21]{BLU07} no longer applies. {  The same difficulty also appears for the stationary equation $-\pair{\bfB x, Du} = \L u + g + D_i f^i$.}
It would be an interesting question whether the techniques in this paper can be extended to this generalized setting.

Furthermore, it is an intriguing question whether we can replace $\L$ in \cref{eq:kol} with nonlinear operators, such as the $p$-Laplace, and whether the theory continues to hold as demonstrated in \cite{DiB93}.

\tocless{\section*{Acknowledgement}}
M.~Hou was supported by [Swedish Research Council dnr: 2019-04098].
The author wishes to thank Benny Avelin for fruitful discussions and revising the paper.

\printbibliography

\end{document}